\documentclass{amsart}
\usepackage[utf8]{inputenc}

\usepackage{mathrsfs}
\usepackage{amscd,amssymb,amsopn,amsmath,amsthm,graphics,amsfonts,enumerate,verbatim,calc}

\newtheorem{theorem}{Theorem}[section]
\newtheorem{lemma}[theorem]{Lemma}

\theoremstyle{definition}
\newtheorem{definition}[theorem]{Definition}

\theoremstyle{remark}
\newtheorem{remark}[theorem]{Remark}

\usepackage[english]{babel}
\usepackage[T1]{fontenc}
\usepackage[utf8]{inputenc}

\usepackage{amsmath, amsthm, amscd, amsfonts}
\usepackage{amssymb, tikz}
\usepackage{mathtools}

\usepackage{hyperref}
\usepackage[capitalise]{cleveref}

\usepackage{setspace}

\DeclareMathOperator{\len}{lh}

\DeclareMathOperator{\dom}{dom}
\DeclareMathOperator{\range}{range}

\DeclareMathOperator{\Suc}{Suc}

\def\MQB{{\mathbb{Q}}}
\def\MRB{{\mathbb{R}}}

\newtheorem{claim}[theorem]{Claim}

\def\MQB{{\mathbb{Q}}}
\def\MRB{{\mathbb{R}}}

\def\rmark{\mbox{$\rm\bf\rule{0.06em}{1.45ex}\kern-0.05em R$}}
\def\pmark{\mbox{$\rm\bf\rule{0.06em}{1.45ex}\kern-0.05em P$}}
\def\nmark{\mbox{$\rm\bf\rule{0.06em}{1.45ex}\kern-0.05em N$}}
\def\vdash{\mbox{$\rm\| \kern-0.13em -$}}
\newcommand{\lusim}[1]{\smash{\underset{\raisebox{1.2pt}[0cm][0cm]{$\sim$}}
{{#1}}}}

\newcommand{\Lim}{\hbox{Lim }}

\newcommand{\rt}{\hbox{rt}}

\usepackage{pb-diagram}

\def\rmark{\mbox{$\rm\bf\rule{0.06em}{1.45ex}\kern-0.05em R$}}
\def\pmark{\mbox{$\rm\bf\rule{0.06em}{1.45ex}\kern-0.05em P$}}
\def\nmark{\mbox{$\rm\bf\rule{0.06em}{1.45ex}\kern-0.05em N$}}
\def\vdash{\mbox{$\rm\| \kern-0.13em -$}}

\begin{document}

\title{On slow minimal reals I}

\author{Mohammad Golshani}
\address{Mohammad Golshani, School of Mathematics, Institute for Research in Fundamental Sciences (IPM), P.O.\ Box:
19395--5746, Tehran, Iran.}
\email{golshani.m@gmail.com}
\thanks{The first author's research has been supported by a grant from IPM (No. 99030417). He also thanks Heike Mildenberger
for her discussions on Shelah's creature forcing and the results of this paper.}

\author{Saharon Shelah}
\address{Einstein Institute of Mathematics, The Hebrew University of Jerusalem, Jerusalem,
91904, Israel, and Department of Mathematics, Rutgers University, New Brunswick, NJ
08854, USA.}
\email{shelah@math.huji.ac.il}
\thanks{ The second
author's research has been partially supported by the European Research Council grant 338821. This is
publication 1198 of second author.}

\thanks{The authors thank the referee of the paper for his comments on early versions of the paper.}

\date{}



\begin{abstract}
Answering a question of Harrington, we show that there exists a proper forcing notion, which adds a minimal real $\eta \in \prod_{i<\omega} n^*_i$, which is eventually different from any old real
in $\prod_{i<\omega} n^*_i$,
where the sequence $\langle n^*_i \mid i<\omega     \rangle$
grows slowly.
\end{abstract}

\maketitle

\section{Introduction}

The method of creature forcing was introduced by Shelah for solving problems related to cardinal invariants like the unbounded
number or the splitting number, as well as questions of the existence of special
kinds of $P$-points. We refer to  \cite{r-shelah} for a more complete history about the development of the subject and its wide applications.

Let us call a real $\bold{r}: \omega \rightarrow \omega$ is slow, if for each $n > 0, \bold{r}(n) \leq n^{g(n)}$ where $g: \omega \rightarrow \omega$ is
non-decreasing with $\lim_{n \rightarrow \infty}g(n)=\infty$
and $\lim_{n \rightarrow \infty} \frac{g(n)}{\log_2(n)}=0$

In this paper, we use the method of tree creature forcing to  answer a question of Leo Harrington\footnote{Harrington asked the question from the second author in personal communication.}. There are several examples of forcing notions like Silver-Prikry forcing \cite{Grigorieff}, Laver forcing \cite{groszek}, Jensen's minimal forcing \cite{jensen}, Rational perfect set forcing \cite{miller}, Sacks forcing \cite{sacks}, splitting forcing \cite{schilhan}    and so on, which add a minimal real into the ground model. In these examples, the real is fast growing and can not be dominated by the ground model reals. On the other hand, given a sequence  $\langle n^*_i \mid i<\omega     \rangle$
 of natural numbers which grows very fast, one can define a forcing notion which adds a real $\eta \in \prod_{i<\omega} n^*_i$, which is minimal and is eventually different from any old real in $\prod_{i<\omega} n^*_i$; see for example  \cite{carl} and \cite{judah-shelah}. Motivated by these results, Harrington asked the second author if there is a proper forcing notion which adds a minimal real as above into a sequence
$\langle n^*_i \mid i<\omega     \rangle$ which grows slowly.
In this paper we   give abstract conditions, called local minimality condition and global minimality condition, such that if a tree creature  forcing notion satisfies them, then it adds a slow minimal real.
In a further paper we prove similar results for the measured creature forcing.

The paper is organized as follows. In Section \ref{Creatures and tree creatures}, we introduce some preliminary results on creature forcing.
Given a sequence $\langle n^*_i \mid i<\omega     \rangle$
 of natural numbers which grows slowly\footnote{See Section \ref{example for minimality} for some examples of such sequences.}, we define, in Section \ref{Minimality with creatures}, a class of creature forcing notions, where each of them adds a real $\eta \in \prod_{i<\omega} n^*_i$, which is minimal and is eventually different from any old real in $\prod_{i<\omega} n^*_i$. A very special case of the results of this section is proved in \cite{ci-shelah}.
In Sections \ref{Sufficient conditions for global minimality condition} and \ref{example for minimality}, we use  probabilistic arguments, to show that the class of such forcing notions is non-empty.

We assume no familiarity with creature forcing, and present all required preliminaries, to make the paper as self contained as possible.

\section{Creatures and tree creatures}
\label{Creatures and tree creatures}
In this section, we briefly review some concepts from Shelah's creature forcing,   that will be used in the rest of the paper. Our presentation follows  \cite{mildenberger} and \cite{r-shelah}.

\begin{definition}
\begin{enumerate}
\item [$(a)$] A quasi-tree $(T, \lhd)$ over $X$ is a set $T \subseteq X^{<\omega}$
with the initial segment relation, such that $T$ has a $\lhd$-minimal element, called the root of $T$, $\rt(T).$
It is called a tree, if it is closed under initial segments of length $\geq \len(\rt(T))$.

\item [$(b)$] For $\eta \in T, \Suc_T(\eta)$ is the set of all
immediate successors of $\eta$ in $T$:
\[
\Suc_T(\eta)=\{ \nu \in T: \eta \lhd \nu \text{~and~} \neg \exists \rho \in T (\eta \lhd \rho \lhd \nu)          \}.
\]

\item [$(c)$] $\max(T)$ is the set of all maximal nodes (if there are any) in $T$:
\[
\max(T)=\{\eta \in T: \neg \exists \rho \in T, \eta \lhd \rho       \}.
\]

\item [$(d)$] $\hat{T}=T \setminus \max(T)$.

\item [$(e)$] For $\eta \in T, T^{[\eta]}=\{\nu \in T: \nu$ and $\eta$ are comparable  $ \}$.

\item [$(f)$] $\Lim(T)$ is the set of all cofinal branches through $T$:
\[
\Lim(T)=\{\eta \in X^\omega: \exists^{\infty}n,~ \eta \restriction n \in T  \}.
\]
\item [$(g)$] $T$ is well-founded if it has no cofinal branches.

\item [$(h)$] a subset $J \subseteq T$ is a front of $T$, if $J$ consists of $\lhd$-incomparable elements and for any branch $\eta$
of $T$, $\eta \restriction n \in J,$ for some $n.$
\end{enumerate}
\end{definition}

Let $\chi$ be a sufficiently large enough regular cardinal.
Let also $\mathbf{H}: \omega \to V \setminus \{\emptyset\}$ be a function
such that
for each $i<\omega,$ $|\mathbf{H}(i)| \geq 2$.
We call $\mathbf{H}(i)$ the reservoir at $i$.
The forcing notions we define aim to add a function $g \in \prod_{i<\omega} \mathbf{H}(i).$
\begin{definition}
\begin{enumerate}
\item [$(a)$] A (weak) creature for $\mathbf{H}$ is a tuple $t=(\mathbf{nor}[t], \mathbf{val}[t], \mathbf{dis}[t])$, where
\begin{itemize}
\item [(1)] $\mathbf{nor}[t] \in \mathbb{R}^{\geq 0}$.

\item [(2)] $\mathbf{val}[t] \subseteq \bigcup_{m_0 < m_1 < \omega} \{ (w, u) \in  \prod_{i< m_0}\mathbf{H}(i) \times \prod_{i<m_1}\mathbf{H}(i):   w \lhd u          \}$.

\item [(3)] $\mathbf{dis}[t] \in H(\chi).$
\end{itemize}
The family of all  creatures for $\mathbf{H}$ is denoted by $\text{CR}[\mathbf{H}]$.
\item [$(b)$] If $t \in \text{CR}[\mathbf{H}],$ then $\text{basis}(t)= \dom(\mathbf{val}[t])$
and $\text{pos}(t)=\range(\mathbf{val}[t])$.

\item [$(c)$] A creature $t$ is a tree-creature, if $\text{basis}(t)=\{\eta\}$ is a singleton and no distinct elements of $\text{pos}(t)$ are $\lhd$-comparable.
The family of all  tree-creatures for $\mathbf{H}$ is denoted by $\text{TCR}[\mathbf{H}]$. We also set
\[
\text{TCR}_{\eta}[\mathbf{H}] =\{ t \in \text{TCR}[\mathbf{H}]:      \text{basis}(t)=\{\eta\}        \}.
\]
\end{enumerate}
\end{definition}

\begin{definition}
\begin{enumerate}
\item [$(a)$] Let $K \subseteq \text{CR}[\mathbf{H}]$.
A function $\Sigma: [K]^{\leq \omega} \to \mathcal{P}(K)$ is called a sub-composition operation on $K$, if
\begin{itemize}

\item [(1)] $\Sigma(\emptyset)=\emptyset$, and for each $s \in K, s \in \Sigma(s).$

\item [(2)] (transitivity) If $\mathcal{S} \in [K]^{\leq \omega}$ and for each $s \in \mathcal{S}, \mathcal{S}_s \in [K]^{\leq \omega}$
is such such that $s \in \Sigma(\mathcal{S}_s)$, then $\Sigma(\mathcal{S}) \subseteq \Sigma(\bigcup_{s \in \mathcal{S}}\mathcal{S}_s)$.

\end{itemize}
\item [$(b)$] A pair $(K, \Sigma)$ is called a creating pair for $\mathbf{H}$ if $K \subseteq \text{CR}[\mathbf{H}]$
and $\Sigma$ is a sub-composition operation on $K$

\item [$(c)$] A sub-composition operation $\Sigma$ on $K \subseteq \text{TCR}[\mathbf{H}]$ is called a tree-composition operation on $K$. Such a pair
$(K, \Sigma)$ is called a tree-creating pair for $\mathbf{H},$ if in addition:
\begin{itemize}
\item [(1)] If $\mathcal{S} \in [K]^{\leq \omega}$ and $\Sigma(\mathcal{S}) \neq \emptyset,$ then there exists a well-founded quasi tree $T \subseteq \bigcup_{n<\omega} \prod_{i<n}\mathbf{H}(i)$ and a sequence $\langle s_\nu: \nu \in \hat{T}   \rangle \subseteq K$
    such that $\mathcal{S}=\{    s_\nu: \nu \in \hat{T}    \}$, and for each $\nu \in \hat{T}, s_\nu \in \text{TCR}_{\nu}[\mathbf{H}]$
    and $\text{pos}(s_\nu)=\Suc_T(\nu)$.

\item [(2)] If $t \in \Sigma(s_\nu: \nu \in \hat{T}),$ then $t \in \text{TCR}_{\rt(T)}[\mathbf{H}]$
and $\text{pos}(t) \subseteq \max(T)$ (where $T$ is as in (1)).
\end{itemize}
\end{enumerate}
\end{definition}
To each tree-creating pair $(K, \Sigma)$, we assign the forcing notion $\MQB^{\text{tree}}(K, \Sigma)$ as follows:
\begin{definition}
Assume $(K, \Sigma)$ is  a tree-creating pair for $\mathbf{H}$. Then $\MQB^{\text{tree}}(K, \Sigma)$
consists of all sequences $p= \langle \mathfrak{c}^p_\eta: \eta \in T^p     \rangle$, where
\begin{enumerate}
\item $T^p \subseteq \bigcup_{n<\omega} \prod_{i<n}\mathbf{H}(i)$ is a tree with no maximal nodes, i.e., $\max(T^p)=\emptyset.$

\item $\mathfrak{c}^p_\eta \in \text{TCR}_{\eta}[\mathbf{H}] \cap K$ and $\text{pos}(\mathfrak{c}^p_\eta)=\Suc_{T^p}(\eta)$.
\end{enumerate}
Given $p, q \in \MQB^{\text{tree}}(K, \Sigma)$, we set $p \leq q$ ($p$ is an extension of $q$), if $T^p \subseteq T^q$
and for each $\eta \in T^p,$ there exists a well-founded quasi-tree $T \subseteq (T^p)^{[\eta]}$ such that $\mathfrak{c}^p_\eta \in \Sigma(\mathfrak{c}^q_\nu: \nu \in \hat{T}).$
\end{definition}
The forcing notion $\MQB^{\text{tree}}(K, \Sigma)$
 is maximal among all  forcing notions that we assign to  $(K, \Sigma)$, in the sense that all other forcing notions that we define, are subsets of
$\MQB^{\text{tree}}(K, \Sigma)$.

\section{Minimality with creatures}
\label{Minimality with creatures}
In this section, we define a class of creature forcing notions which add a minimal real with slow splitting.
\begin{definition}
Let $(K, \Sigma)$ be  a tree-creating pair for $\mathbf{H}$.  The forcing notion $\MQB^*(K, \Sigma)$ consists of all conditions
$p= \langle \mathfrak{c}^p_\eta: \eta \in T^p     \rangle \in \MQB^{\text{tree}}(K, \Sigma)$ such that
\begin{enumerate}
\item  $T^p \subseteq \omega^{<\omega}$ is a finite branching tree with  $\max(T^p)=\emptyset.$

\item  Each $\mathfrak{c}^p_\eta$ is finitary, i.e., $|\mathbf{val}[\mathfrak{c}^p_\eta]|$ is finite.

\item  $\eta \in \Lim(T^p) \implies \lim\inf (\mathbf{nor}[\mathfrak{c}^p_{\eta \restriction n}]: n \geq \len(\rt(T^p)))=\infty.$

\item  Each $\mathfrak{c}^p_\eta$ is $2$-big, i.e., if $\mathbf{nor}[\mathfrak{c}^p_\eta] \geq 1$ and $\text{pos}(\mathfrak{c}^p_\eta)= u_0 \cup u_1$ is a partition
of $\text{pos}(\mathfrak{c}^p_\eta)$, then for some $\mathfrak{d} \in \Sigma(\mathfrak{c}^p_\eta), \mathbf{nor}[\mathfrak{d}] \geq \mathbf{nor}[\mathfrak{c}^p_\eta] - 1$
and for some $l<2, \text{pos}(\mathfrak{d}) \subseteq u_l.$
\end{enumerate}
Given $p, q \in \MQB^*(K, \Sigma),$ $p \leq q$ iff
\begin{enumerate}
\item  $T^p \subseteq T^q.$
\item  For each $\eta \in T^p, \mathfrak{c}^p_\eta \in \Sigma(\mathfrak{c}^q_\eta)$.
\end{enumerate}
\end{definition}
A proof of the next lemma can be found in  \cite{mildenberger} and \cite{r-shelah}, and is essentially based on the fact that all creatures involved in the forcing are $2$-big.
\begin{lemma} \label{lem:basic properties}
\begin{enumerate}

 \item [(a)] $(\MQB^*(K, \Sigma), \leq)$ is proper.

\item [(b)] (continuous reading of names) For any condition $p \in \MQB^*(K, \Sigma),$ any $k<\omega$ and any $\MQB^*(K, \Sigma)$-name $\lusim{f}: \omega \to \check{V}$,
there exists $q$ such that
\begin{itemize}
\item [$(\alpha)$] $q \leq_k p,$ i.e., $q \leq p$
and if $\eta \in T^q$ and $\len(\eta) \leq k \vee \mathbf{nor}[\mathfrak{c}^q_\eta] \leq k$ then $\mathfrak{c}^q_\eta=\mathfrak{c}^p_\eta$,
so $\Suc_{T^q}(\eta)=\Suc_{T^p}(\eta)$.

\item [$(\beta)$] For every $n$, there exists a front $J_n$ of $T^q$ such that
\[
\eta \in J_n \implies \exists a \in V, q^{[\eta]} \Vdash \text{``}\lusim{f}(n)= a\text{''}.
\]
\end{itemize}
\end{enumerate}
\end{lemma}

We now state and prove the main result of this section, which gives  sufficient conditions for the forcing
notion $\MQB^*(K, \Sigma)$  to add a minimal real.
\begin{theorem} \label{minimal real}
Assume  $\MQB^*(K, \Sigma)$  satisfies the following conditions:
\begin{enumerate}

 \item [(a)] (the local minimality condition) If $\mathbf{nor}[\mathfrak{c}^p_\eta] \geq 1$
and $E$ is an equivalence relation on $\text{pos}(\mathfrak{c}^p_\eta)$, then for some $\mathfrak{d} \in \Sigma(\mathfrak{c}^p_\eta)$ we have
\begin{itemize}
\item [$(\alpha)$] $\mathbf{nor}[\mathfrak{d}] \geq \mathbf{nor}[\mathfrak{c}^p_\eta]-1.$
\item [$(\beta)$] $E \restriction \text{pos}(\mathfrak{d})$ is trivial: either every equivalence class is a singleton or it has one equivalence class.
\end{itemize}

\item [(b)] (the global minimality condition) For every $p \in \MQB^*(K, \Sigma)$ and $k<\omega$ there are $q \in \MQB^*(K, \Sigma)$ and $m > k, \len(\rt(T^p))$
such that
\begin{itemize}
\item [$(\alpha)$] $q \leq_k p.$

\item [$(\beta)$] If $J$ is a front of $T^q$ consisting of sequences of length $> m,$ $S$ is a finite set and $\langle f_\rho: \rho\in J  \rangle$
is a sequence of one-to-one functions from $\Suc_{T^q}(\rho)$ into $S$, then there exists a partition $S=S_0 \cup S_1$
of $S$ such that for every $\eta \in T^q \cap \omega^m,$ there are $r_0, r_1$
such that each $r_l \leq_{k}q^{[\eta]}$
and
\[
\eta \lhd \rho \in J \cap T^{r_l} \implies f_\rho``[\Suc_{T^{r_l}}(\rho)] \subseteq S_l.
\]
\end{itemize}
\end{enumerate}
Then $\Vdash_{\MQB^*(K, \Sigma)}$``$\lusim{\eta}=\bigcup \{\rt(T^p): p \in \dot{G}_{\MQB^*(K, \Sigma)}    \}$ is a minimal real''.
\end{theorem}
\begin{proof}
 Assume $p \in \MQB^*(K, \Sigma)$ and  $p \Vdash_{\MQB^*(K, \Sigma)}$``$\lusim{f}$ is a new real''. Without loss of henrality, we may assume that
  $p \Vdash_{\MQB^*(K, \Sigma)}$``$\lusim{f} \in$$^{\omega}2$.
We show that there is $q \leq p,$ such that $q \Vdash_{\MQB^*(K, \Sigma)}$``$\lusim{\eta} \in V[\lusim{f}]$''.

By Lemma \ref{lem:basic properties}(a) and \cite{r-shelah} and extending $p$ if necessary, we can assume that there exists a sequence $\langle (J_n, H_n): n < \omega \rangle$
such that:
\begin{enumerate}
\item $J_n$ is a front of $T^p$.
\item $J_{n+1}$ is above $J_n$, i.e., $\forall \nu \in J_{n+1} \exists l<\len(\nu),~ \nu \restriction l \in J_n.$
\item $H_n: J_n \to \mathbf{H}(n)$ is such that for each $\nu \in J_n, p^{[\nu]} \Vdash_{\MQB^*(K, \Sigma)}$``$\lusim{f}(n)=H_n(\nu)$''.
\end{enumerate}
\begin{claim}
\label{3.5A}
For every $k<\omega$ and $p' \leq p$ there are $q, \bar{J}, \bar{f}$ and $\bar{n}$
such that:
\begin{enumerate}
\item[(a)] $q \leq_k p'$,

\item[(b)] $\bar{J}=\langle J_l: l<\omega \rangle$, where each $J_l$ is a front of $T^q$,

\item[(c)] $J_{l+1}$ is above $J_l$,

\item[(d)] $\bar{n}=\langle n(l): l<\omega \rangle$,

\item[(e)] $\bar{f}=\langle f_\nu: \nu \in \bigcup_{l<\omega} J_l     \rangle$,

\item[(f)] If $\nu \in J_l,$ then $f_\nu: \Suc_{T^q}(\nu) \to$$^{n(l)}2$ is such that:
\begin{itemize}
\item $f_\nu$ is one-to-one,

\item if $\rho \in \Suc_{T^q}(\nu), \nu \in J_l$, then
\[
q^{[\rho]} \Vdash_{\MQB^*(K, \Sigma)}\text{``}\lusim{f}\restriction n(l) = f_\nu(\rho)\text{''}.
\]
\end{itemize}
\end{enumerate}
\end{claim}
\begin{proof}
To start, let $q' \leq_k p'$ and $m > k,  \len(\rt(T^{p'}))$ witness (b)$(\beta)$.
Let $n_0$ be such that
\begin{itemize}
\item $\nu \in T^{q'} \wedge \len(\nu) \geq n_0 \implies \mathbf{nor}[\mathfrak{c}^{q'}_\nu] > k +|T^{q'} \cap \omega^{m}|.$

\item $n_0 > m.$

\item $J_{n_0}$ is above $T^{q'} \cap \omega^{m}$.
\end{itemize}
For every $n> n_0,$ we define a function $H^+_n$ with domain a subset of
$\Lambda_n$
and values in $^{n}2$,
a sequence $  \langle \mathfrak{d}^n_\nu: \nu \in \Lambda_n     \rangle$,   by downward induction on $\len(\nu)$
as follows:

\underline{Case 1. $\nu \in J_n:$}
Let $\rho=H^+_n(\nu) \in$$^{n}2$ be such that
$$l<n \wedge \eta \in J_l \wedge \eta \unlhd \nu \implies \rho(l)=H_l(\eta)$$
and set $\mathfrak{d}^n_\nu=\mathfrak{c}^{q'}_\nu$.

\underline{Case 2. $\Suc_{T^{q'}}(\nu) \subseteq \Lambda_n$ and $\langle H^+_n(\rho), \mathfrak{d}^n_\rho: \rho \in \Suc_{T^{q'}}(\nu)  \rangle$
is defined:}
Define an equivalence relation $E_\nu$ on $\Suc_{T^{q'}}(\nu)$ by
\[
(\rho_1, \rho_2) \in E_\rho \iff H^+_n(\rho_1) = H^+_n(\rho_2).
\]
By the local minimality condition (a),
 there exists $\mathfrak{d}^n_\nu \in$
$\Sigma(c^{q'}_\nu)$ such that $\mathbf{nor}[\mathfrak{d}_\nu] \geq \mathbf{nor}[c^{q'}_\nu]-1$
and $H^+_n \restriction \text{pos}(\mathfrak{d}_\nu)$ is constant or one-to-one. If $H^+_n \restriction \text{pos}(\mathfrak{d}_\nu)$ is constant,  then choose such a $\mathfrak{d}^n_\nu$
and let $H^+_n(\nu)$ be that constant value. Otherwise let $H^+_n(\nu)$ be undefined.
So  $\dom(H^+_n)$  is an upward
closed subset of $\Lambda_n$
which includes $J_n$.

Next we show that there exists $n<\omega$ such that there is no $\eta \in T^{q'} \cap \omega^{n_0}$ with $\eta \in \dom(H^+_n).$
Suppose otherwise. Then as $T^{q'} \cap \omega^{n_0}$ is finite, there is  $\eta \in T^{q'} \cap \omega^{n_0}$
such that $\exists^{\infty} n > n_0,~ \eta \in \dom(H^+_n).$
Let $D$ be an ultrafilter on $\omega$
such that $\{n:     \eta \in \dom(H^+_n)   \} \in D.$ Let
\begin{center}
$T=\{\rho \in T^{q'}: \eta \lhd \rho$ and $\forall l (\len(\eta) \leq l < \len(\rho) \implies \{ n: \rho \restriction (l+1) \in \text{pos}(\mathfrak{d}^n_{\rho \restriction l})       \} \in D)               \}.$
\end{center}
Let $q'' \leq q'$ be such that $T^{q''}=T$ and for every $\nu \in T,$ $\mathfrak{c}^{q''}_\nu = \mathfrak{d}^n_\nu$.
Then $q'' \Vdash$``$\lusim{f} \in V$''. To see this, let $\rho \in \Lim(T)$. Then,
\begin{center}
$q'' \Vdash$``$\lusim{f} \restriction (n_0, \omega) = \{(n, H^+_n(\rho \restriction n)): n > n_0         \}$'',
\end{center}
from which the result follows.
We get a contradiction by the fact that $f$ is a new real.

So we can find $n_1> m$ such that $\dom(H^+_{n_1}) \cap (T^{p'} \cap$ $\omega^{n_0})=\emptyset.$ By  extending $q'$, let us assume that
for each $n \geq n_1$ and $\nu \in \dom(H^+_n)\setminus J_n, \Suc_{T^{q'}}(\nu)=\text{pos}(\mathfrak{d}_\nu^n).$ It then follows that, for each $n \geq n_1,$
\[
\nu \in T^{q'} \cap \dom(H^+_n) ~\&~ \nu \lhd \nu' \in T^{q'} \cap \Lambda_n \implies H^+_n(\nu)=H^+_n(\nu').
\]
For $l<\omega$ set $n(l)=n_1+l$ and
\[
J_l=\{\rho \in \Lambda_{n(l)}: \rho \notin \dom(H^+_{n(l)}) \text{~but~} \Suc_{T^{q'}}(\rho) \subseteq   \dom(H^+_{n(l)})           \}.
\]
Then $J_l$ is a front of $T^{q'}$ above $T^{q'} \cap$ $\omega^{n_0}$ and by our construction,
for every $\rho \in J_l, $  $H^+_{n(l)} \restriction \text{pos}(\mathfrak{d}^{n(l)}_\rho)$ is one-to-one.

For $l<\omega$ and $\nu \in J_l$ set $f_\nu=H^+_{n(l)} \restriction \text{pos}(\mathfrak{d}^{n(l)}_\nu)$.
Then $f_\rho: \Suc_{T^{q'}}(\nu) \to$$^{n(l)}2$ is a one-to-one function.

Finally let $q \leq_k p'$ be such that  for each $l<\omega$ and $\rho \in J_{n(l)},$ $\Suc_{T^{q'}}(\rho)=\text{pos}(\mathfrak{d}^{n(l)}_\rho)$.
\end{proof}
\begin{claim}
\label{3.5B}
Suppose $p, \lusim{f}$ are as above, $k, i< \omega$ and $p' \leq p.$ There there are $r,  \Psi, n_\ast$ and $n_\bullet$ such that:
\begin{enumerate}
\item[(a)] $r \leq_k p'$,

\item[(b)] $\Psi: \prod_{l< n_\bullet} \mathbf{H}(l) \cap T^q \to \{0, 1\}$

\item[(c)] If $\nu \in T^r$, $\len(\nu)=n_\ast$, then
\[
r^{[\nu]} \Vdash_{\MQB^*(K, \Sigma)} \text{``} \lusim{\eta}(i)=\Psi(\lusim{f} \restriction n_\bullet)\text{''}.
\]
\end{enumerate}
\end{claim}
\begin{proof}
Let $p', k, i$ be given as above. Without loss of generality we may assume that $k>i$. By the global minimality condition \ref{minimal real}(b), applied
to $p'$ and $k+1$,  we can find $q \leq_{k+1} p'$ and $m> k+1, \len(\rt(T^{p'}))$ satisfying clauses ($\alpha$)
and($\beta$) there. Let $\bar{J}$ and $\bar{f}$ be as in the conclusion of Lemma \ref{3.5A}. So for some $l, J=T^q \cap J_l$
is a front of $T^q$ above level $m$. Let $S=$$^{n(l)}2$. Hence we can apply \ref{minimal real}(b)($\beta$)
and get $(S_0, S_1, \langle r_{\nu, j}: \nu \in T^q$ with $\len(\nu)=m,~ j<2    \rangle)$  as there.
Let $r$ be a $k$-extension of $p'$ such that
\[
T^r=\bigcup\{ T^{r_{\nu, j}}: \nu \in T^q \cap~^{m}\omega \text{~and~} j=\nu(l)                            \}.
\]
Set also $n_\ast=\max\{\len(\nu): \nu \in J   \}$
and $n_\bullet=m.$
Finally define  $\Psi: \prod_{l< n_\bullet} \mathbf{H}(l) \cap T^q \to \{0, 1\}$ by
\[
\Psi(t)=j \iff j=\nu(l) \text{~where~}\nu \text{~is such that~} t \in  T^{r_{\nu, j}}.
\]
It is easily seen that $r, n, \Psi, n_\ast$ and $n_\bullet$ are as required.
\end{proof}
\begin{claim}
\label{3.5C}
For every $k$ and $p' \leq_k p$, there are $q, \bar{n}$ and $\bar{\Psi}$ such that
\begin{enumerate}
\item[(a)] $q \leq_k p,$

\item[(b)] $\bar{n}=\langle n(l): l<\omega   \rangle$,

\item[(c)] $\bar{\Psi}=\langle  \Psi_l: l<\omega    \rangle$,

\item[(d)] $\Psi_l$ is a function from $T^q \cap$$^{n(l)}\omega \to 2,$

\item[(e)] If $\nu \in$$T^q \cap$$^{n(l)}\omega$, then
\[
q^{[\nu]} \Vdash_{\MQB^*(K, \Sigma)}\text{``} \lusim{\eta}(l)=\Psi_l(\lusim{f} \restriction n(l))\text{''}.
\]
\end{enumerate}
\end{claim}
\begin{proof}
By Lemma \ref{3.5B}  we can find  $\bar{r},  \bar{\Psi}, \bar{n_\ast}$ and $\bar{n_\bullet}$ such that:
\begin{itemize}
\item $\bar{r}=\langle r_l: l<\omega  \rangle$
is such that $r_0 \leq_k p'$ and for each $l<\omega,$ $r_{l+1} \leq_{k+l} r_l$,

\item $\bar{n_\ast}= \langle n_{\ast, l}: l<\omega    \rangle$ and $\bar{n_\bullet}= \langle n_{\bullet, l}: l<\omega    \rangle$
are increasing sequences of natural numbers,

\item $\bar{\Psi}= \langle \Psi_l: l<\omega     \rangle$, where $\Psi_l: \prod_{j< n_{\bullet, l}} \mathbf{H}(l) \cap T^q \to \{0, 1\}$,

\item  If $\nu \in T^{r_l}$, $\len(\nu)=n_{\ast, l}$, then
\[
r_l^{[\nu]} \Vdash_{\MQB^*(K, \Sigma)} \text{``} \lusim{\eta}(l)=\Psi(\lusim{f} \restriction n_{\bullet, l})\text{''}.
\]
\end{itemize}
For $l<\omega$ set $n(l)=n_{\ast, l}$ and let $q=\lim_{l \to \infty}r_l$ be the natural condition obtained by the fusion argument applied to the sequence $\bar{r}.$
\end{proof}
Finally we are ready to complete the proof of Theorem \ref{minimal real}. Suppose that $p \Vdash_{\MQB^*(K, \Sigma)}$``$\lusim{f} \in$$^{\omega}2$''.
If $p \nVdash_{\MQB^*(K, \Sigma)}$``$\lusim{f}\notin V$'', then for some $q \leq p, q \Vdash$``$\lusim{f}\in \check{V}$'' and we are done. So suppose otherwise. By Lemma
\ref{3.5C}, there are $q \leq p$, $\bar{n}$ and $\bar{\Psi}$ as there. So for every $l<\omega,$
\[
q \Vdash_{\MQB^*(K, \Sigma)}\text{``}\lusim{\eta}(l)=\Psi_l(\lusim{f}\restriction n(l))\text{''}.
\]
It follows that $q \Vdash_{\MQB^*(K, \Sigma)}$``$\lusim{\eta} \in V[\lusim{f}]$''. We are done.
\end{proof}

\section{Sufficient conditions for global minimality condition}
\label{Sufficient conditions for global minimality condition}
In this section we give some sufficient conditions to guarantee the global minimality condition
of Theorem \ref{minimal real}.
\begin{definition}
Given a finite set $S$, we define a probability space $(\Omega, \mathcal{F}, \text{Prob})$ as follows:
\begin{itemize}
\item $\Omega=\{(S_0, S_1): S_0 \cup S_1$ is a partition of $S      \}$.

\item $\mathcal{F}=\mathcal{P}(\Omega)$.

\item For $A \subseteq \Omega,$ $\text{Prob}(A)= $$\dfrac{|A|}{|\Omega|}$.
\end{itemize}
\end{definition}
\begin{lemma}
\label{conditions fo rglobal minimality}
Assume that for every $p=\langle  \mathfrak{c}^p_\eta: \eta \in T^p \rangle \in \MQB^*(K, \Sigma)$ and every $k<\omega$
there are $m> k, \len(\rt(T^p))$ and $q$ such that:
\begin{itemize}
\item [(a)] $q \leq_k p.$

\item [(b)] If $\eta \in T^p$ and $\len(\eta)\geq m,$ then $\mathbf{nor}[\mathfrak{c}^q_\eta] > k+ |T^q \cap$ $\omega^m|$.

\item [(c)] If $J$ is a front of $T^q$ is such that for $\eta \in J,$   $\len(\eta) \geq m$, $S$ is a finite set and $\langle f_\rho: \rho\in J  \rangle$
is a sequence of one-to-one functions from $\Suc_{T^q}(\rho)$ into $S$, then we can find a sequence
$\langle a_\eta: \eta \in \Lambda \rangle$, where
\[
\Lambda = \{ \eta \in T^q: \len(\eta) \geq m \text{~and~} \eta \text{~is below~}J \text{~or is in ~}J  \}
\]
such that:
\begin{enumerate}
\item [$(\alpha)$] $a_\eta \in (0, 1)_{\mathbb{R}}$.

\item [$(\beta)$] If $\eta \in J,$ then
$a_\eta \geq \text{Prob} ($for some $l \in \{0, 1 \},$ there is no $\mathfrak{d} \in \Sigma(\mathfrak{c}^p_\eta)$
with $\mathbf{nor}[\mathfrak{d}] \geq \mathbf{nor}[\mathfrak{c}^p_\eta]-1$ and $f_\eta``[\text{pos}(\mathfrak{d})] \subseteq S_l).$

\item [$(\gamma)$] If $\eta \in \Lambda \setminus J,$ then
$a_\eta \geq \text{Prob} ($there is no $\mathfrak{d} \in \Sigma(\mathfrak{c}^p_\eta)$
with $\mathbf{nor}[\mathfrak{d}] \geq \mathbf{nor}[\mathfrak{c}^p_\eta]-1$ such that $\nu \in \text{pos}(\mathfrak{d}) \implies E_\nu$ does not occur),
whenever $E_\nu$ are events for $\nu \in \text{pos}(\mathfrak{c}^q_\eta)$ each of probability $\leq a_\nu.$

\item [$(\delta)$] $1 > \Sigma\{a_\eta: \eta \in T^q \cap$ $\omega^m    \}$.
\end{enumerate}
\end{itemize}
Then the global minimality condition
of Theorem \ref{minimal real} holds.
\end{lemma}
\begin{proof}
For $\eta \in \Lambda,$ let $\text{Good}_\eta$ be the event: ``For $l<2,$ there is $r_l \in \MQB^*(K, \Sigma)$ such that
$\rt(T^{r_l})=\eta, r_l \leq_k q^{[\eta]}$ and $\forall \nu \in T^{r_l} \cap J,~ \Suc_{T^{r_l}}(\nu) \subseteq S_l$''.
\begin{claim}
For $\eta \in \Lambda, \text{Prob}(\text{Good}_\eta) \geq 1 - a_\eta.$
\end{claim}
\begin{proof}
We prove the claim by downward induction on $\len(\eta)$. Suppose $\eta \in J.$  Let $\Theta_\eta$ denote the statement ``for some $l \in \{0, 1 \},$ there is no $\mathfrak{d} \in \Sigma(\mathfrak{c}^p_\eta)$
with $\mathbf{nor}[\mathfrak{d}] \geq \mathbf{nor}[\mathfrak{c}^p_\eta]-1$ and $f_\eta``[\text{pos}(\mathfrak{d})] \subseteq S_l)$''. So by clause $(c)(\beta),$
$a_\eta \geq \text{Prob} (\Theta_\eta)$.  This implies
\[
1 - a_\eta \leq 1- \text{Prob}(\Theta_\eta) \leq \text{Prob}(\text{Good}_\eta).
\]
Otherwise, $\eta \in \Lambda \setminus J$ and  $\eta$ is not $\lhd$-maximal in $T^q.$
Let $\Phi_\eta$ denote the statement ``there is no $\mathfrak{d} \in \Sigma(\mathfrak{c}^p_\eta)$
with $\mathbf{nor}[\mathfrak{d}] \geq \mathbf{nor}[\mathfrak{c}^p_\eta]-1$ such that $\nu \in \text{pos}(\mathfrak{d}) \implies E_\nu$ does not occur''.
By clause $(c)(\gamma),$
$a_\eta \geq \text{Prob} (\Phi_\eta)$,
so
\[
1 - a_\eta \leq 1- \text{Prob}(\Phi_\eta) \leq \text{Prob}(\text{Good}_\eta).
\]
\end{proof}
Since $1 > \Sigma\{a_\eta: \eta \in T^q \cap$ $\omega^m    \}$, we can find a pair $(S_0, S_1) \in \Omega$
such that for every $\eta \in T^q \cap$ $\omega^m, \text{Good}_\eta$ occurs and hence we can choose $r_0, r_1$ as guaranteed by
$\text{Good}_\eta$ and we are done.
\end{proof}

\begin{lemma} \label{corollary to conditions fo rglobal minimality}
Assume for a dense set of conditions $q \in \MQB^*(K, \Sigma)$ there exists a sequence $\bar a = \langle a_\eta: \eta \in T^q \rangle$
such that:
\begin{itemize}
\item [(a)] Each $a_\eta \in (0, 1]_{\MRB}$.

\item [(b)] If $\eta \in \Lim(T^q),$ then $\lim_n \langle a_{\eta \restriction n}: n \geq \len(\rt(T^q))     \rangle=0.$

\item [(c)] For every large enough $m$, we have
\begin{enumerate}
\item [$(\alpha)$] $1 > \Sigma\{a_\eta: \eta \in T^q \cap$ $\omega^m    \}$.

\item [$(\beta)$] For any  $\eta \in T^q \cap$ $\omega^m$, if $\langle E_\nu: \nu \in \Suc_{T^q}(\eta)     \rangle$
is a sequence of events each of probability $\leq a_\nu,$ then we have $a_\eta \geq \text{Prob} ($there is no $\mathfrak{d} \in \Sigma(\mathfrak{c}^p_\eta)$
with $\mathbf{nor}[\mathfrak{d}] \geq \mathbf{nor}[\mathfrak{c}^p_\eta]-1$ such that $\nu \in \text{pos}(\mathfrak{d}) \implies E_\nu$ does not occur).
\end{enumerate}
\end{itemize}
Then the global minimality condition
of Theorem \ref{minimal real} holds.
\end{lemma}
\begin{proof}
By Lemma \ref{conditions fo rglobal minimality}.
\end{proof}

\section{A forcing notion satisfying the conditions of Theorem \ref{minimal real}}
\label{example for minimality}
In this section we introduce a forcing notion which satisfies the conditions of local and global minimality, hence it adds a minimal real.

 Let
$h: \omega \to \omega$ be a non-decreasing function with $\lim h(n)=\infty$ (e.g. $n \mapsto \log_2(\log_2 (n))$ or $n \mapsto \log_*(n)$,
where $\log_*(n)$ is defined by recursion as  $\log_*(0)=\log_*(1)=1$,
 $\log_*(\beth_{m+1}(2))=\log_*(\beth_m(2))+1$, for $m<\omega,$\footnote{Where $\beth_m(2)$ is defined inductively as $\beth_0(2)=1$
and $\beth_{m+1}(2)=2^{\beth_m(2)}$.} and for each $n$ with $\beth_m(2) \leq n < \beth_{m+1}(2)$,
 $\log_*(n)=\log_*(\beth_m(2))$).

Let $\mathbf{H}:\omega \to V$ be defined by
$$\mathbf{H}(n)=(\max\{1, n\})^{h(n)}.$$ Let
$$K=\{ t\in TCR[\mathbf{H}]: \mathbf{nor}[t]=  \log_2(\log_2(|\text{pos}(t)|))       \}.$$
Also let $\Sigma: [K]^{\leq \omega} \to \mathcal{P}(K)$
be such that it is non-empty only for singletons, and   for every $t \in K,$
$$\Sigma(t) =\{ s \in K: \mathbf{val}[s] \subseteq    \mathbf{val}[t]            \}.$$
Clearly $(K, \Sigma)$ forms a tree creating pair.
Let $\MQB_h$ be defined as follows:
\begin{itemize}
\item [(A)] $p \in \MQB_h$ if
\begin{enumerate}
\item $p= \langle \mathfrak{c}^p_\eta: \eta \in T^p \rangle \in \MQB^*(K, \Sigma)$.
\item $T^p \subseteq \bigcup_{n<\omega} \prod_{i<n} \mathbf{H}(i)$ is a tree, such that for each non-maximal node $\eta \in T^p, |\Suc_{T^p}(\eta)| \geq 4$.
\item For $\eta \in T^p, \mathbf{nor}[\mathfrak{c}^p_\eta]=\mathbf{nor}(\Suc_{T^p}(\eta))= \log_2(\log_2(|\Suc_{T^p}(\eta)|))$.
\item If $\eta \in \Lim(T^p)$ and $i<\omega,$ then $\lim_n \dfrac{\mathbf{nor}(\Suc_{T^p}(\eta))}{n^i}$$=\infty.$
\end{enumerate}

\item [(B)] $p \leq q$ iff $p \leq_{\MQB^*(K, \Sigma)} q.$
\end{itemize}
\begin{lemma}
The forcing notion $\MQB_h$ satisfies the local and global minimality conditions of Theorem \ref{minimal real}.
\end{lemma}
\begin{proof}
\underline{$\MQB_h$ satisfies the local minimality condition:}
Suppose $\mathbf{nor}[\mathfrak{c}^p_\eta] \geq 1$ and $E$ is an equivalence relation on $\text{pos}(\mathfrak{c}^p_\eta)$.

Consider the set $X=\{[\rho]_E: \rho \in \text{pos}(\mathfrak{c}^p_\eta)  \}$.
If there exists $\rho \in \text{pos}(\mathfrak{c}^p_\eta)$ such that $[\rho]_E$ has size  $\geq \sqrt{|\text{pos}(\mathfrak{c}^p_\eta)|}$,
then set
\[
\mathfrak{d} = \{ (\eta, \nu): \nu \in [\rho_*]_E           \}.
\]
Clearly,   $\mathbf{val}[\mathfrak{d}] \subseteq    \mathbf{val}[\mathfrak{c}^p_\eta]$    and so $\mathfrak{d} \in \Sigma(\mathfrak{c}^p_\eta)$. Also,
$E \restriction \mathfrak{d}$ is trivial. We also have

$\hspace{2.cm}$ $\mathbf{nor}[\mathfrak{d}] = \log_2(\log_2(|\text{pos}(\mathfrak{d})|))$

$\hspace{3.1cm}$ $\geq  \log_2(\log_2(\sqrt{|\text{pos}(\mathfrak{c}^p_\eta)|}))$

$\hspace{3.1cm}$ $\geq \log_2(\log_2(|\text{pos}(\mathfrak{c}^p_\eta)|)) - \log_2 2$

$\hspace{3.1cm}$ $\geq \log_2(\log_2(|\text{pos}(\mathfrak{c}^p_\eta)|)) -1$

$\hspace{3.1cm}$ $= \mathbf{nor}[\mathfrak{\mathfrak{c}^p_\eta}]-1$.

So we are done.
Otherwise, each equivalence class
$[\rho]_E$ has size less than $\sqrt{|\text{pos}(\mathfrak{c}^p_\eta)|}$. But then $|X| \geq \sqrt{|\text{pos}(\mathfrak{c}^p_\eta)|},$
so pick $t_\rho \in [\rho]_E$ for each $\rho \in \text{pos}(\mathfrak{c}^p_\eta)$ and set
\[
\mathfrak{d} = \{ (\eta, t_\rho): \rho \in \text{pos}(\mathfrak{c}^p_\eta)       \}.
\]
Again,  $\mathbf{val}[\mathfrak{d}] \subseteq    \mathbf{val}[\mathfrak{c}^p_\eta]$    and so $\mathfrak{d} \in \Sigma(\mathfrak{c}^p_\eta)$. Also,
$E \restriction \mathfrak{d}$ is trivial. As before,
\begin{center}
 $\mathbf{nor}[\mathfrak{d}] = \log_2(\log_2(|\text{pos}(\mathfrak{d})|)) \geq \mathbf{nor}[\mathfrak{\mathfrak{c}^p_\eta}]-1$,
\end{center}
and so we are done.
The result follows.

\underline{$\MQB_h$ satisfies the global minimality condition:}
We check Lemma \ref{corollary to conditions fo rglobal minimality}.
Let $p \in \MQB_h$ and $k<\omega.$ Then we can find $q \in \MQB_h$, $m_* > k, \len(\rt(T^p))$ and $m_0$ such that:
\begin{itemize}
\item $q \leq_k p.$
\item $k < m_0 < m_*.$
\item $\eta \in T^q \wedge \mathbf{nor}[\mathfrak{c}^q_\eta] \leq k \implies \len(\eta) < m_0.$
\item If $n \geq m_0,$ then $\langle  |\Suc_{T^q}(\eta)|: \eta \in T^q \cap$ $\omega^n    \rangle$
is constant.
Moreover, for some non-decreasing function $h_1 \leq h$, $h_1: [m_0, \omega) \to \omega \setminus \{0\},$
we have $ |\Suc_{T^q}(\eta)|=\len(\eta)^{2^{h_1(\len(\eta))}}$.
\item $h_1(m_*) > k+ |T^q \cap$ $\omega^{m_0}|$.
\end{itemize}
For $\eta \in T^q$, we define $b_\eta$ by
\begin{center}
$b_\eta=\text{Prob}$(for a partition $S_0 \cup S_1$ of $m_\eta=|\text{pos}(\mathfrak{c}^q_\eta)|$, we have $\bigvee_{l<2} |S_l| < \sqrt{m_\eta}).$
\end{center}
Then it is clear that
\begin{center}
$b_\eta \leq \dfrac{2\cdot \Sigma\{m_\eta: i <  \sqrt{m_\eta}   \}}{2^{m_\eta}}$
\end{center}
(the 2 is because of $\bigvee_{l<2}$ and the sum is because of $\bigvee_{i<\sqrt{m_\eta} }|S_l=i|$).
So
\[
b_\eta \leq \dfrac{(m_\eta)^{\sqrt{m_\eta}}}{2^{m_\eta}} = \dfrac{2^{\sqrt{m_\eta}\cdot \log_2(m_\eta)}}{ 2^{m_\eta}}.
\]
Thus for some $n_*,$
\[
\eta \in T^q, \len(\eta)=n \geq n_* \implies b_\eta \leq 2^{\sqrt{n^{2^{h_1(n)}}}\cdot 2^{h_1(n)}\cdot \log_2(n)}\cdot 2^{-(n^{2^{h_1(n)}})} \leq 2^{-(n^{2^{h_1(n)}/2^4})}.
\]
For $\eta \in T^q, \len(\eta) \geq n_*$ let
\[
a_\eta = \Sigma\{ 2^{-(n^{2^{h_1(n)}}/3)}: n \geq \len(\eta)        \}.
\]
Then $a_\eta \leq 2^{-(n^{2^{h_1(n)}})}$, where $n=\len(\eta)$, because
\[
n^{2^{h_1(n)}} +1 \leq (n+1)^{2^{h_1(n)}} \leq (n+1)^{2^{h_1(n+1)}}.
\]
We show that  $a_\eta$'s, $\eta \in T^q$ are as required. It is clear that each $a_\eta \in (0,1)_{\MRB}$.
Also, if $\eta \in \Lim(T^q),$ then clearly
$\lim_n \langle a_{\eta \restriction n}: n \geq \len(\rt(T^q))     \rangle = 0$ ( as $a_\eta \leq 2^{-(n^{2^{h_1(n)}})}$).
Also, for $n \geq n_*$, we have
\[
\Sigma\{a_\eta: \eta \in T^q \cap \omega^n   \} \leq \Sigma\{2^{-(n^{2^{h_1(n)}})}: \eta \in  T^q \cap \omega^n  \} = |T^q \cap \omega^n|\cdot 2^{-(n^{2^{h_1(n)}})} < 1.
\]
Finally clause (c)$(\beta)$ of Lemma \ref{corollary to conditions fo rglobal minimality} holds by the choice of $b_\eta$'s and $a_\eta$'s.
\end{proof}
\begin{remark}
We could also use the function $h(n)=\log_*(\log_*(n))$ and the norm $\mathbf{nor}[t]=  \log_*(\log_*(|\text{pos}(t)|))$.
In this case, there is no need to require the tree $T^p$ satisfies the extra property ``for each non-maximal node $\eta \in T^p, |\Suc_{T^p}(\eta)| \geq 4$''.
\end{remark}

\end{document}